\documentclass[11pt]{amsart}
\usepackage{inputenc}
\usepackage[hmarginratio={1:1},
  vmarginratio={1:1}, vmargin = 2.5cm, hmargin = 2cm]{geometry}

\usepackage{graphicx} 
\usepackage{subcaption}
\usepackage[export]{adjustbox}
\usepackage{amsmath}
\usepackage{amssymb}
\usepackage{mathrsfs}
\usepackage{xcolor}
\usepackage{algorithm}
\usepackage{algpseudocode}
\usepackage{svg}
\usepackage{url}
\usepackage{tabularx}
\usepackage{multicol}
\usepackage{epstopdf}

\newtheorem{theorem}{Theorem}[section]

\numberwithin{equation}{section}

\newcommand{\Sp}{\mathbb S}
\newcommand{\R}{\mathbb R}

\begin{document}

\title{Improved bounds for the double cap conjecture}

\author{Domonkos Czifra, \'Akos D\'ucz, M\'at\'e Matolcsi, D\'aniel Varga, P\'al Zs\'amboki}

\address{Czifra Domonkos, HUN-REN Alfr\'ed R\'enyi Institute of Mathematics, Re\'altanoda utca 13-15, 1053 Budapest, Hungary}
\email{czifra.domonkos@renyi.hu}

\address{\'Akos D\'ucz, HUN-REN Alfr\'ed R\'enyi Institute of Mathematics, Re\'altanoda utca 13-15, 1053 Budapest, Hungary}
\email{akos@renyi.hu}

\address{M\'at\'e Matolcsi, HUN-REN Alfr\'ed R\'enyi Institute of Mathematics, Re\'altanoda utca 13-15, 1053 Budapest, Hungary
and 
Department of Analysis and Operations Research,
Institute of Mathematics,
Budapest University of Technology and Economics,
M\H uegyetem rkp. 3., H-1111 Budapest, Hungary
}
\email{matomate@renyi.hu}

\address{D\'aniel Varga, HUN-REN Alfr\'ed R\'enyi Institute of Mathematics, Re\'altanoda utca 13-15, 1053 Budapest, Hungary}
\email{daniel@renyi.hu}

\address{P\'al Zs\'amboki, HUN-REN Alfr\'ed R\'enyi Institute of Mathematics, Re\'altanoda utca 13-15, 1053 Budapest, Hungary}
\email{zsamboki.pal@renyi.hu}

\date{March 2026}

\begin{abstract}

In 1974, Witsenhausen asked for the maximum possible density $\alpha_n$ of a measurable subset $A$ of the unit sphere $\mathbb{S}^{n-1}\subset \mathbb{R}^n$ such that $A$ contains no pair of orthogonal vectors. For $n=3$, the best known lower bound is $1 - 1/\sqrt{2} = 0.29289\dots$, obtained from the natural ``double cap'' construction of two opposite spherical caps. This construction was conjectured by Gil Kalai \cite{kalai} to be optimal for all $n$. In this paper, we use a novel approach to establish the upper bound $\alpha_3\le 0.2953$, improving on the previous best known bound of $0.2977$ due to Bekker et al. \cite{bekker}. Our approach combines arguments from harmonic analysis with the concept of the geometric fractional chromatic number of finite graphs, introduced recently by Ambrus et al. \cite{ambrus}. Within this framework, any finite subset of the sphere yields an upper bound for $\alpha_n$; we obtain our bound by identifying a suitable 33-element point set through a large-scale computer search.

The same method can also be used in higher dimensions to yield potential improvements of the best known bounds. 
\end{abstract}


\maketitle

\section{Introduction}

A measurable subset $A$ of the unit sphere $\Sp^{n-1} = \{ x\in  \R^n : \ \|x\| = 1 \}$ is said to {\it avoid orthogonality} if $\langle a, a'\rangle \ne 0$ for all $a, a'\in A$.  Witsenhausen \cite{wits} posed the problem of finding the maximum possible density of such a set $A\subset \Sp^{n-1}$. That is, we aim to determine the value of
\begin{equation}\label{an}
\alpha_n = \sup \{ \omega(A)/\omega_n \ : \ A \subseteq \Sp^{n-1} \ \textrm{is \  measurable \ and \ avoids \ orthogonality}\},
\end{equation}
where $\omega$ is the surface measure of 
the sphere, and $\omega_n$ is the total measure of the sphere.

\medskip

The best known lower bound for $\alpha_n$ comes from the natural construction of two open antipodal spherical caps of spherical radius $\pi/4$. This 
``double cap'' construction is conjectured to be optimal for all $n$ by Gil Kalai \cite{kalai}. The conjecture is open for all $n\ge 3$. If true, it would imply that $\alpha_n=(\sqrt{2}+o(1))^{-n}$.

\medskip

As for upper bounds, Witsenhausen \cite{wits} observed the trivial bound $\alpha_n\le 1/n$ by considering an orthonormal  system of $n$ unit vectors, and an averaging argument over the sphere. The first exponentially decreasing bound, $\alpha_n\le (1+o(1))1.13^{-n}$ was given by Frankl and Wilson \cite{fs}, which was later improved  to $\alpha_n\le (1+o(1))1.225^{-n}$ by Raigorodskii \cite{rai}. 

\medskip

Besides progress on the asymptotics of $\alpha_n$, there have been several improvements on upper bounds in low dimensions. For $n=3$, the first improvement on the trivial bound $\alpha_3\le 1/3$ was obtained by a harmonic analytic version of Delsarte's linear programming bound combined with some combinatorial constraints by DeCorte and Pikhurko \cite{dp}, who proved $\alpha_3\le 0.313$. The bound was later improved to $\alpha_3\le 0.30153$ by DeCorte, Oliveira and Vallentin \cite{dov} by incorporating complete positivity constraints, and most recently to $\alpha_3\le 0.2977$ by Bekker, Oliveira, Kuryatnikova and Vera \cite{bekker} by similar methods. The main result of this note is to further improve the upper bound on $\alpha_3$ by a novel approach of combining  harmonic analytic arguments with the notion of the geometric fractional chromatic number of finite graphs introduced by Ambrus et al.  \cite{ambrus}. 

\begin{theorem}\label{main}
For the Witsenhausen problem in dimension 3 we have $\alpha_3\le 0.2953$.
\end{theorem}

\medskip

The structure of the paper is the following. In Section \ref{sec2}, we describe how the geometric fractional chromatic number (GFCN) of a finite set $X\subset \Sp^{n-1}$ fits into the framework of upper bounding $\alpha_n$. In Section \ref{sec3}, we link up the notion of GFCN with the harmonic analytic version of Delsarte's linear programming bound, and show how weak duality can produce a verifiable dual witness to our upper bound. This gives the method of the proof of Theorem \ref{main}.  Finally, the Appendix contains essential auxiliary ingredients, such as the description of the computer search which resulted in an appropriate finite set $X\subset \Sp^{n-1}$ consisting of 33 antipodal pairs, and the steps of the symbolic verification of our dual witness.

\section{The geometric fractional chromatic number}
\label{sec2}

The notion of the geometric fractional chromatic number, introduced in \cite{ambrus}, is a very effective tool in several problems with underlying  symmetries. In this section, we describe how this concept fits into the framework of Witsenhausen's problem. 

\medskip

We first make a technical remark, and introduce some terminology. If $A\subset \Sp^{n-1}$ avoids orthogonality, then so does the symmetrized set $A\cup -A$, of possibly larger measure. As such, we will assume without loss of generality, throughout the paper, that $A$ is symmetric, $A=-A$. Similarly, in the forthcoming arguments, whenever we consider a finite subset $X\subset \Sp^{n-1}$, we always assume that $X$ is symmetric, $X=-X$. As such, $X$ can be written as a union of antipodal pairs, $X=\cup_{i=1}^N \{x_i, -x_i\}$. (That is, $X$ is essentially a subset of the projective plane, but still, we would rather think of it as a collection of antipodal pairs on the surface of the sphere.) To such a set $X$, we will associate a corresponding {\it orthogonality graph} $G_X$ in the following way: any vertex $v_i$ of $G_X$ is a pair $\{x_i, -x_i\}$, and we join two vertices $v_i$ and $v_j$  by an edge if and only if $\langle x_i, x_j \rangle =0$. For an orthogonal transformation $T\in O(n)$ we say that $T(v_i)=v_j$ if $T(x_i)=x_j$ or $T(x_i)=-x_j$. We will call two sets of vertices $Y=\{v_{i_1}, \dots, v_{i_k}\}$ and $Z=\{v_{j_1}, \dots, v_{j_k}\}$ {\it geometrically congruent}, in notation $Y\cong Z$, if there exists an orthogonal transformation $T$ of $\R^n$ such that $T(Y)=Z$. 

\medskip

After this preparation, let a symmetric finite set $X=\cup_{i=1}^N \{x_i, -x_i\}\subset \Sp^{n-1}$, $|X|=2N$, and a symmetric, orthogonality avoiding set $A\subset \Sp^{n-1}$ with density $\delta(A)=\frac{\omega(A)}{\omega_n}$ be given, and let $\mu$ denote the normalized Haar measure on the orthogonal group $O(n)$. Notice that if $A$ avoids orthogonality and $T$ is an orthogonal transformation of $\R^n$, then the intersection $T(A)\cap X=\cup_{j=1}^k \{ x_{i_j}, -x_{i_j} \}$ is necessarily a collection of antipodal pairs, which corresponds to an independent set $I=\{v_{i_1}, \dots, v_{i_k}\}$ in $G_X$, i.e. a subset of vertices without any edges among them. With a slight abuse of notation we will denote this by $T(A)\cap G_X=I$. As such, it is natural to introduce a probability measure on the collection of independent sets of $G_X$ in the following way. For any independent set of vertices $I\subset G_X$, we define the {\it atomic density}

\begin{equation}\label{ax}
a_X(I)= \mu\Big (T\in O(n): T(A)\cap G_X=I\Big ).    
\end{equation}

We have the following trivial properties of the atomic densities $a_X(I)$: 
\begin{equation}\label{iep}
a_X(I) \ge 0 \ \ {\textrm{for all \ }} I    
\end{equation}
\begin{equation}\label{iet}
\sum_I a_X(I)= 1,    
\end{equation}
i.e., the atomic densities indeed constitute a probability measure. 

\medskip

Also, for any fixed index $i$, the density of $A$ can be expressed as the $\mu$-probability that $v_i\in T(A)$, that is $\mu(T\in O(n):v_i\in T(A))$. This, in turn, is the sum of atomic densities for independent sets $I$ containing $v_i$. Therefore we conclude  
\begin{equation}\label{ie1}
\sum_{v_i\in I} a_X(I)= \delta(A) \ \ {\textrm{for all \ }} i=1, 2, \dots N.    
\end{equation}
In particular, the value of the sum above is the same value for any vertex: 
\begin{equation}\label{ie1b}
\sum_{v_i\in I} a_X(I)= \sum_{v_j\in I} a_X(I) \ \ {\textrm{for all \ }} 1\le i, j\le N.    
\end{equation}

There is one more crucial property of the atomic densities $a_X(I)$, corresponding to the underlying geometric invariance of the problem. Namely, if two sets of vertices $Y,Z\subset G_X$ are geometrically congruent, then the shift-invariance of the Haar measure ensures that $\mu\Big( T\in O(n): Y\subset T(A)\Big)=\mu\Big( T\in O(n): Z\subset T(A)\Big)$. This can be expressed via the atomic densities $a_X(I)$ as follows: 
\begin{equation}\label{iec}
\sum_{I\supset Y} a_X(I)-\sum_{I\supset Z} a_X(I)=0  \ \ {\textrm{whenever \ }} Y\cong Z.
\end{equation}
Note that equation \eqref{ie1b} is a special case of \eqref{iec}, with the choice $Y=\{v_i\}, Z=\{v_j\}$.

Having equations \eqref{iep}, \eqref{iet}, \eqref{ie1}, \eqref{iec}  at hand, we can give an upper bound on $\alpha_n$ by the following linear program: 

\begin{equation}\label{gfcn}
{\textrm{maximize}} \sum_{v_1\in I} a_X(I) \ {\textrm{subject to}} 
\end{equation}
\begin{equation*}
a_X(I) \ge 0 \ \ {\textrm{for all \ }} I, 
\end{equation*}
\begin{equation*}
\sum_I a_X(I)= 1, 
\end{equation*}
\begin{equation*}
\sum_{I\supset Y} a_X(I)-\sum_{I\supset Z} a_X(I)=0  \ \ {\textrm{whenever \ }} Y\cong Z \ {\textrm{in}} \ G_X .
\end{equation*}
Note that we did not include equation \eqref{ie1b} among the constraints because it is formally included in \eqref{iec}. 

\medskip

Let $\Gamma(G_X)$ denote the objective value of the linear program above. In analogy with the terminology in \cite{ambrus} we define the geometric fractional chromatic number $\chi_{gf}(G_X)$ of the graph $G_X$ as  $\chi_{gf}(G_X)=1/\Gamma(G_X)$. 

\medskip

By the discussion above, we have the following upper bound on $\alpha_n$: 

\begin{equation}\label{da}
\alpha_n\le \inf\{ \Gamma(G_X) \ :\ X\subset\Sp^{n-1}, \ X \ {\textrm{is finite, symmetric}}\}.
\end{equation}
It would be interesting to know whether equality holds in \eqref{da}.

\medskip

It is worth mentioning (cf. a similar discussion in \cite{mrvz}) that $\Gamma(G_X)$ is monotonically decreasing with respect to inclusion, i.e. 
\begin{equation}\label{mon}
\Gamma(G_{X_1})\le \Gamma(G_{X_2}) \ {\textrm{for all}} \  G_{X_1}\supset G_{X_2},
\end{equation}
although we will not use this property in this note. 

\medskip

In practice, we only have heuristic guidance as to what properties the set $X$ should possess in order to produce a low value of $\Gamma(G_X)$. First, $X$ should contain many geometric congruences, so as to induce several constraints of the type \eqref{iec}. Second, $X$ should contain many orthogonal pairs $\langle x_i, x_j \rangle =0$, so that the number of independent sets (i.e. the number of variables in the linear program) remains as low as possible. 

Using a computer search described in Subsection \ref{ss:computer search}, we have found a 21-point set $X$ that achieves the bound $\alpha_3 \le \Gamma(G_X)=32/107\approx0.2991$. See Table \ref{tab:GFCN X21 coordinates}.

\section{Harmonic analysis and weak duality}
\label{sec3}

In this section, we will link up the linear program \eqref{gfcn}  with the harmonic analytic version of Delsarte's linear programming bound to produce a sharper upper bound on $\alpha_n$. 

\medskip

Let $t\in [-1, 1]$, and $x,y\in \Sp^{n-1}$ be given such that $\langle x, y \rangle =t$. For any measurable set $A\subset \Sp^{n-1}$ define the function $f_A: [-1,1]\to \R$, as 
\begin{equation}\label{fa}
f_A(t)=\mu\{T\in O(n): \  Tx\in A \  {\textrm{and}} \  Ty\in A\}
\end{equation}
It is clear that the value $f_A(t)$ does not depend on the particular choice of $x$ and $y$. Also, it is easy to see (e.g. \cite[Lemma 4]{dp}) that $f_A$ is continuous and {\it positive definite}. As such, Schoenberg's theorem (\cite{sch}, \cite[Theorem 14.3.3]{dx}) applies, and $f_A$ can be written in the form 
\begin{equation}\label{pdf}
f_A(t)=\sum_{k=0}^\infty c_kP_k(t),  
\end{equation}
where $c_k\ge 0$, and $P_k$ are the Gegenbauer polynomials $C^{(n-2)/2}_k$ of index $(n-2)/2$ and degree $k$, with the normalization $P_k(1)=1$. In particular, for $n=3$, $P_k$ are the Legendre polynomials. 

\medskip

Assume now that $A\subset \Sp^{n-1}$ is symmetric, avoids orthogonality and has density $\delta(A)$. In this case, $f_A$ has the following essential properties: 
\begin{equation}\label{fpos}
f_A(t)\ge 0 \ {\textrm{for all}} \ t\in[-1,1]    
\end{equation}

\begin{equation}\label{cpos}
c_k\ge 0 \ {\textrm{in}} \  \eqref{pdf}    
\end{equation}

\begin{equation}\label{fsym}
f_A(t)=f_A(-t) \ {\textrm{due to}} \ A=-A    
\end{equation}

\begin{equation}\label{codd}
c_k=0 \ {\textrm{for all odd}} \ k  \ {\textrm{due to }} \ \eqref{fsym}     
\end{equation}

\begin{equation}\label{t1}
\sum_{2|k, k=0}^\infty c_k=\delta(A)  \ {\textrm{due to}} \ t=1 \ {\textrm{in}} \ \eqref{pdf}   
\end{equation}

\begin{equation}\label{t0}
\sum_{2|k, k=0}^\infty c_kP_k(0)=0  \ {\textrm{due to}} \ t=0 \ {\textrm{in}} \ \eqref{pdf},  
\end{equation}
the last equation being true because $A$ avoids orthogonality. 

Furthermore, if $X\subset \Sp^{n-1}$ is a finite, symmetric set, and $G_X$ is the corresponding graph introduced in Section \ref{sec2}, then the atomic densities $a_X(I)$ and the coefficients $c_k$ are linked via the following relation: 
\begin{equation}\label{axck}
\sum_{2|k, k=0}^\infty c_kP_k(\langle v_i, v_j \rangle) - \sum_{v_i,v_j\in I} a_X(I) =0 \ {\textrm{for all}} \ 1\le i\le j\le N.
\end{equation}
This last equation needs a little clarification. Recall that $v_i=\{x_i, -x_i\}$, and hence the value of $\langle v_i, v_j \rangle$ is ambiguous. Nevertheless, the value $P_k(\langle v_i, v_j \rangle)$ is well-defined, because the polynomial $P_k$ is an even function for even values of $k$. Thus, the first sum above is the $\mu$-probability of $\{x_i, x_j\}$ being contained in $T(A)$.  The second sum is the $\mu$-probability of $\{v_i, v_j\}=\{x_i, -x_i, x_j, -x_j\}$ being contained in $T(A)$. Due to the symmetry of $X$ and $A$ these two events are the same. 

\medskip

Note that equality $i=j$ is allowed in \eqref{axck}, in which case $P_k(\langle v_i, v_i\rangle)=1$, and thus \eqref{axck} reduces to \eqref{t1} via \eqref{ie1} . 

\medskip

In summary, we obtain an upper bound on $\alpha_n$ by the following (infinite dimensional) linear program: 

\begin{equation}\label{lp2}
{\textrm{maximize}} \  \sum_{k=0}^\infty c_k \ {\textrm{subject to}} \ 
\eqref{fpos},\eqref{cpos},  \eqref{codd}, \eqref{t0}, \eqref{axck}, \eqref{iep}, \eqref{iet}, \eqref{iec}.
\end{equation}

Let $\Delta(G_X)$ denote the objective value (supremum) of the linear program above. By the discussion above, we have the following upper bound on $\alpha_n$:
\begin{equation}\label{pa}
\alpha_n\le \inf\{ \Delta(G_X) \ :\ X\subset\Sp^{n-1}, \ X \ {\textrm{is finite, symmetric}}\}.
\end{equation}

We conjecture that equality holds in \eqref{pa}. Similar convergence results have been proved in \cite{bekker}, but it is not clear whether those results apply to our setting without modification. 

\medskip

Finally, we apply weak duality to the linear program \eqref{lp2} in order to obtain a verifiable witness for an upper bound on $\alpha_n$. 

\begin{theorem}\label{bound}
Assume $X=\{x_1, -x_1, \dots, x_N, -x_N\}\subset \Sp^{n-1}$ is a finite, symmetric subset of the sphere, $v_i=\{x_i, -x_i\}$, $G_X$ is the orthogonality graph corresponding to $X$, and real parameters $q_1$, $s_{Y,Z}$ (for subsets $Y$ and $Z$ of $G_X$, where $Y\cong Z$), $w_0$ and $w_{i,j}$ (for all $1\le i\le j\le N$)  are given such that the inequalities 
\begin{equation}\label{forck}
w_0P_k(0) + \sum_{i\le j}w_{i,j}P_k(\langle v_i, v_j \rangle )\ge 1 \ {\textrm{for all}} \ k\ge 0, \ k \  {\textrm{is even}} 
\end{equation}

\begin{equation}\label{forax}
q_1 - \sum_{v_i, v_j\in I}w_{i,j} + \sum_{I\subset Y}s_{Y,Z} -\sum_{I\subset Z}s_{Y,Z}\ge 0 \ {\textrm{for all}} \ I\subset X, I \  {\textrm{is independent}}
\end{equation}
are satisfied. Then $\alpha_n\le q_1$.    
\end{theorem}
\begin{proof}
The proof is a direct application of linear duality, as follows. Let us use the notation 
$\beta_k=w_0P_k(0)+\sum_{i\le j}w_{i,j}P_k(\langle v_i, v_j \rangle )$, and $\gamma_I=q_1 - \sum_{v_i, v_j\in I}w_{i,j} + \sum_{I\subset Y}s_{Y,Z} -\sum_{I\subset Z}s_{Y,Z}$. Consider the weighted sum of equations \eqref{iet}, \eqref{iec}, \eqref{t0} and \eqref{axck} with respective weights $q_1$, $s_{Y,Z}$, $w_0$ and $w_{i,j}$. This weighted sum gives us $\sum_{2|k, k=0}^\infty \beta_k c_k +\sum_{I\subset X}\gamma_Ia_X(I)= q_1$. Therefore, using the assumptions $\beta_k\ge 1, \gamma_I\ge 0$, equation \eqref{t1}, and the nonnegativity of the variables $c_k$ and  $a_X(I)$ we obtain 
\begin{equation}\label{last}
\delta(A)=\sum_{2|k, k=0}^\infty  c_k \le \sum_{2|k, k=0}^\infty \beta_k c_k +\sum_{I\subset X}\gamma_Ia_X(I) =q_1.       
\end{equation}
\end{proof}

The reader might notice that we have not used the constraint \eqref{fpos} in the formulation and proof of Theorem \ref{bound}. It is simply because, in our experience, this constraint is always satisfied automatically in practice, and therefore does not carry any value in sharpening the bound on $\alpha_n$. 

\medskip

While the proof of Theorem \ref{bound} is straightforward, it is far from trivial to make use of it in practice. The difficulty is twofold. First and foremost, one needs to construct a finite, symmetric set $X\subset \Sp^{n-1}$ which has the potential to give a sharp upper bound on $\alpha_n$ via Theorem \ref{bound}. We achieve this by a large-scale computer search governed by some heuristics. Details of this search are described in the Appendix. 

\medskip

Second, once $X$ is fixed, we need to find the parameters  $q_1$, $s_{Y,Z}$, $w_0$ and $w_{i,j}$ as in Theorem \ref{bound}, and verify that the constraints \eqref{forck} and \eqref{forax} are satisfied. The numerical values of the parameters are found by a linear programming solver, and then rounded to rational numbers for purposes of rigorous verification. The verification of \eqref{forax} is fairly easy as there are only finitely many inequalities, and therefore any minor violation of an inequality can be compensated by increasing the value of $q_1$ slightly. However, \eqref{forck} is a collection of infinitely many constraints, and hence the verification procedure is more involved in this case. Notice that as $k\to \infty$, we have  $P_k(\langle v_i, v_j \rangle ) \to 0$ for all $i\ne j$, while $P_k(\langle v_i, v_i \rangle )=P_k(1)=1$ for all $i$ and $k$. Therefore, as $k\to \infty$, the left hand side of \eqref{forck} converges to $\sum_iw_{ii}$. If this sum is larger than 1, then inequality \eqref{forck} is satisfied  automatically above a threshold $k>k_0$. In order to determine the threshold $k_0$, we use standard estimates for the absolute value of Jacobi polynomials (see e.g \cite{dov}):
\begin{equation}\label{integral}
|P_k(\langle v_i, v_j \rangle )|=|P_k(\cos \theta )|\le \frac{2}{\pi}\int_0^{\pi/2} (1-\sin^2 \theta \cos^2 \phi )^{k/2}d\phi.
\end{equation}
Still, for any $k<k_0$ the verification of \eqref{forck} requires the evaluation of $P_k(\langle v_i, v_j \rangle )$ to a high precision. This task is non-trivial in itself, as $k$ can range up to several thousands. The way out is to use the recursion formula (see \cite{szego})
\begin{equation}\label{recursion}
P_k(u)=\frac{2k-1}{k}uP_{k-1}(u)-\frac{k-1}{k}P_{k-2}(u).    
\end{equation}
The details of the rigorous verification process are given in the Appendix. 

\medskip

Finally, we provide here the finite set $X$ which is used in the proof of Theorem \ref{main}. 

\medskip

{\it Proof of Theorem \ref{main}.} The upper bound $\alpha_3\le 0.2953$ is witnessed by an appropriate finite, symmetric set $X\subset \Sp^{n-1}$, geometric congruences $Y\cong Z$ among some subsets of $G_X$,  and corresponding parameters $q_1, s_{Y,Z}, w_0$ and $w_{i,j}$ satisfying inequalities \eqref{forck} and \eqref{forax} as described in Theorem \ref{bound}. The witness $X$ consists of 33 antipodal pairs, and hence $|G_X|=33$. By definition, the vectors in $X$ are normalized, $\|x_i\|=1$. However, in order to reduce clutter, we present the coordinates of $x_i$ without normalization: all coordinates are of the form $a + \sqrt{5}b: a, b \in Z$. We provide the coordinates of one point from each antipodal pair in Table~\ref{coord}. As for the congruences $Y\cong Z$, and the coefficients parameters $q_1, s_{Y,Z}, w_0$ and $w_{i,j}$, we refer the reader to the supplementary material \cite{supplementary}, where the rigorous verification procedure is documented. \qed

\begin{figure}[htbp]
    \centering
    \includegraphics[width=0.9\linewidth]{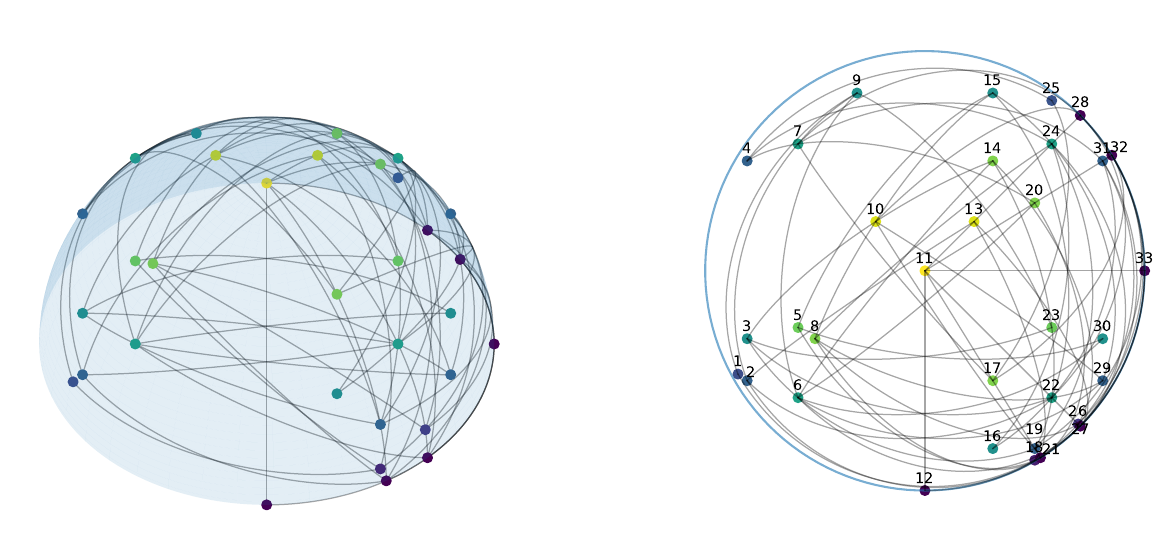}
    \caption{The 33-vertex graph $G_X$ visualized. Left figure slightly elevated viewpoint, right figure orthogonal projection onto the $xy$-plane. Numbers correspond to row indices in Table~\ref{coord}, edges indicate orthogonality, and colors represent the $z$-coordinate value.}
    \label{fig:best33}
\end{figure}


\section{Acknowledgements}

The authors are grateful to Frank Vallentin and Fernando Oliveira de Filho for pointing out the possibility of using the method of GFCN in the context of the Witsenhausen problem. 

M.M. and D.Cz. were supported by grant NKFIH-154121. Á.D., D.V., and P.Zs. were supported by grant NKFIH-153165. D.Cz., D.V., and P.Zs. were supported by the Ministry of Innovation and Technology NRDI Office within the framework of the Artificial Intelligence National Laboratory (RRF-2.3.1-21-2022-00004).

\section{Appendix}

\subsection{Computer search}\label{ss:computer search}

Below we give a high level overview of our graph search. 

\medskip 

The symmetric witness set $X$, for which the value $\Delta(G_X)$ provides the upper bound for $\alpha_3$ as described in Section \ref{sec2}, is obtained by a heuristic search algorithm. Throughout the search we restrict our attention to symmetric sets $X$, and consider any antipodal pair $\{x_i, -x_i\}$ as a single vertex $v_i$ in the orthogonality graph $G_X$, as explained in Section \ref{sec2}.  

\medskip

We start from a promising configuration $G_{X_{0}}$ of small cardinality, and iteratively add and remove vertices until we obtain a sufficiently low value of $\Delta(G_X)$. Below, we refer to  candidate sets $X$ as \emph{nodes}, as in nodes of a search tree, not to be confused with \emph{vertices}, which term is reserved for the elements of $G_X$.

\medskip

Computing $\Delta(G_X)$ is challenging: writing up the linear program requires finding all the independent sets of $G_X$ and all the isometries among them. We work in the space of antipodal pairs, and adapt the isometry checks accordingly. The size of the LP seems to grow exponentially in $|G_X|$, so computing $\Delta(G_X)$ for large graphs is not feasible. We employed the GUROBI solver~\cite{gurobi} on a computer with 1TB of RAM for our graph search, which limited us to circa $|G_X| \leq 33$.

\medskip

The design of the search algorithm was guided by the following considerations:
\begin{itemize}
    \item monotonicity: $\Delta(G_X)$ is monotone w.r.t. set inclusion.
    \item evaluation time: evaluating $\Delta(G_X)$ is roughly exponential in $|G_X|$.
\end{itemize}

\subsubsection{Beam-constrained Monte Carlo Tree Search}

In standard MCTS setups typical in game-tree search \cite{coulom, kocsisszepesvari, silver}, inner nodes have similar evaluation times rather than depending heavily on tree depth.
Hence, we modify the standard MCTS algorithm with a beam search aspect: we extend a node only in the case when its value is at least the k-th best among nodes of the same size.
This trade-off may exclude some promising nodes, but significantly reduces computational demands.

The search starts from a carefully chosen 6-vertex graph $G_{X_0}$ (corresponding to 6 antipodal pairs), see Table~\ref{tab:Sphere_init_coordinates}. We only add vertices that are orthogonal to at least two vertices already present. Consequently, there are $O(n^2)$ possible expansions, which can be computed efficiently by taking cross products of vertex pairs.
Adding new vertices can make others redundant. Therefore, node expansion involves removing vertices as well as adding them. These reduced nodes are inserted into the selection path: they become the parent of the expanded node and the (virtual) child of its grandparent.

Node selection is performed according to the UCB bound \cite{kocsisszepesvari}.
As our target is a simple maximization task and not the standard 2-player min/max-setup, we are looking for the best descendant of the root.
Therefore, during the expansion step, we evaluate all child nodes and backpropagate the maximum value observed among them.
We empirically normalize and rescale target values to the range $[0,1]$.

In the search that found $X$, we use exploration rate 0.1 and beam size 5. The total computational budget is roughly 14 CPU core-years on AMD EPYC 7643 processors.

\subsection{Symbolic verification}

The vertices of $G_X$ are given by closed algebraic expressions, and hence orthogonal pairs and geometric congruences of subsets $Y, Z\subset G_X$ can be determined unambiguously. 

\medskip

After declaring rational values of the parameters $q_1, s_{Y,Z}, w_0$ and $w_{i,j}$ (which we find by linear programming, and slight modifications to account for numerical errors), we must verify that \eqref{forck} and \eqref{forax} are fulfilled. 

\medskip

Necessary tricks for checking \eqref{forck}: 

\medskip

First we find a threshold index $k_0$. For this we use the trivial estimate 
$$|w_0P_k(0)+\sum_{i < j}w_{i,j}P_k(\langle v_i, v_j \rangle )|\le |w_0P_k(0)|+\sum_{i < j}|w_{i,j}| | P_k(\langle v_i, v_j \rangle )|,$$
and for each term we use the estimate 
$$|P_k(\langle v_i, v_j \rangle )|=|P_k(\cos \theta_{i,j} )|\le \frac{2}{\pi}\int_0^{\pi/2} (1-\sin^2 \theta_{i,j} \cos^2 \phi )^{k/2}d\phi.$$

For each fixed $\theta$ the integral is monotonically decreasing in $k$, so we first aim to find a particular $k_0$ such that 
\begin{equation}\label{k0}
\frac{2|w_0|}{\pi} \int_0^{\pi/2} (1-\cos^2 \phi )^{k_0/2}d\phi+
\frac{2}{\pi}\sum_{i< j}|w_{i,j}| \int_0^{\pi/2} (1-\sin^2 \theta_{i,j} \cos^2 \phi )^{k_0/2}d\phi\le \sum_i w_{ii}-1.
\end{equation}
In order to find such a $k_0$, we observe that for each fixed $\theta$ and $k$, the function $(1-\sin^2 \theta \cos^2 \phi )^{k/2}$ is monotonically increasing (as a function of $\phi$) on $[0, \pi/2]$, so the integral is easy to estimate by a Riemannian sum.  With this, we find the threshold $k_0$. Finally, for any $k\ge k_0$ we automatically have 

$w_0P_k(0)+\sum_{i\le j}w_{i,j}P_{k}(\langle v_i, v_j \rangle )=
\sum_iw_{i,i}+
w_0P_k(0)+\sum_{i< j}w_{i,j}P_{k}(\langle v_i, v_j \rangle )\ge \\\sum_iw_{i,i}-\left (|w_0P_k(0)|+\sum_{i < j}|w_{i,j}| | P_k(\langle v_i, v_j \rangle )|\right )\ge \\
\sum_iw_{i,i}-\left (
\frac{2|w_0|}{\pi} \int_0^{\pi/2} (1-\cos^2 \phi )^{k_0/2}d\phi+
\frac{2}{\pi}\sum_{i< j}|w_{i,j}| \int_0^{\pi/2} (1-\sin^2 \theta_{i,j} \cos^2 \phi )^{k_0/2}d\phi
\right )\ge
1$ 

by equation \eqref{k0}, and hence \eqref{forck} is satisfied for all $k\ge k_0$. 
\medskip

Next, we must check \eqref{forck} for all $k<k_0$. This requires high-precision values of $P_k(\langle v_i, v_j \rangle)$. We obtain these using interval arithmetic: we compute $\langle v_i, v_j \rangle$ to precision $10^{-7000}$ and then propagate intervals through the recursion \eqref{recursion}.

\medskip



\begin{table}[]
    \centering
{\tiny
        \begin{tabular}{|c|c|c|}
        \hline
    
$-2\sqrt{5}$ & $2 - 2\sqrt{5}$ & $-1 + \sqrt{5}$ \\
\hline
$-1 - \sqrt{5}$ & $-2$ & $-1 + \sqrt{5}$ \\
\hline
$-1 - \sqrt{5}$ & $1 - \sqrt{5}$ & $2$ \\
\hline
$-1 - \sqrt{5}$ & $2$ & $-1 + \sqrt{5}$ \\
\hline
$-\sqrt{5}$ & $-1$ & $3$ \\
\hline
$-1$ & $-1$ & $1$ \\
\hline
$-1$ & $1$ & $1$ \\
\hline
$-2$ & $1 - \sqrt{5}$ & $1 + \sqrt{5}$ \\
\hline
$1 - \sqrt{5}$ & $1 + \sqrt{5}$ & $2$ \\
\hline
$-1$ & $1$ & $2 + \sqrt{5}$ \\
\hline
$0$ & $0$ & $1$ \\
\hline
$0$ & $-1$ & $0$ \\
\hline
$1$ & $1$ & $2 + \sqrt{5}$ \\
\hline
$-1 + \sqrt{5}$ & $2$ & $1 + \sqrt{5}$ \\
\hline
$-1 + \sqrt{5}$ & $1 + \sqrt{5}$ & $2$ \\
\hline
$-1 + \sqrt{5}$ & $-1 - \sqrt{5}$ & $2$ \\
\hline
$-1 + \sqrt{5}$ & $-2$ & $1 + \sqrt{5}$ \\
\hline
$2\sqrt{5}$ & $-1 - 3\sqrt{5}$ & $3 - \sqrt{5}$ \\
\hline
$2$ & $-1 - \sqrt{5}$ & $-1 + \sqrt{5}$ \\
\hline
$2$ & $-1 + \sqrt{5}$ & $1 + \sqrt{5}$ \\
\hline
$2$ & $-1 - \sqrt{5}$ & $0$ \\
\hline
$1$ & $-1$ & $1$ \\
\hline
$\sqrt{5}$ & $-1$ & $3$ \\
\hline
$1$ & $1$ & $1$ \\
\hline
$\sqrt{5}$ & $3$ & $1$ \\
\hline
$1$ & $-1$ & $-2 + \sqrt{5}$ \\
\hline
$1$ & $-1$ & $0$ \\
\hline
$1$ & $1$ & $0$ \\
\hline
$1 + \sqrt{5}$ & $-2$ & $-1 + \sqrt{5}$ \\
\hline
$1 + \sqrt{5}$ & $1 - \sqrt{5}$ & $2$ \\
\hline
$1 + \sqrt{5}$ & $2$ & $-1 + \sqrt{5}$ \\
\hline
$1 + \sqrt{5}$ & $2$ & $0$ \\
\hline
$1$ & $0$ & $0$ \\
\hline

\end{tabular}
}

    \caption{The homogeneous coordinates of the 33-vertex graph $G_X$ that achieves the upper bound $q_1=0.2953$. We list only one point from each antipodal pair $\{x_i, -x_i\}$}
    \label{coord}
\end{table}

\begin{table}[]
    \centering
{\tiny
\begin{tabular}{|c|c|c|}
\hline
$-1$ & $1$ & $1$ \\
\hline
$1$ & $-1$ & $0$ \\
\hline
$1$ & $-1$ & $1$ \\
\hline
$0$ & $-2$ & $1+\sqrt{5}$ \\
\hline
$-2$ & $1+\sqrt{5}$ & $0$ \\
\hline
$-1-\sqrt{5}$ & $0$ & $2$ \\
\hline
\end{tabular}
}
\caption{Homogeneous coordinates of the initial graph $G_{X_0}$ with six starting vertices.}
    \label{tab:Sphere_init_coordinates}
\end{table}

\begin{table}[]
    \centering
    {\tiny
\begin{tabular}{|c|c|c|}
        \hline
    
$-1 - \sqrt{5}$ & $0$ & $2$ \\
\hline
$-1 - \sqrt{5}$ & $1 - \sqrt{5}$ & $2$ \\
\hline
$-1$ & $1$ & $1$ \\
\hline
$-2$ & $-1 - \sqrt{5}$ & $-1 + \sqrt{5}$ \\
\hline
$-2$ & $1 + \sqrt{5}$ & $-1 + \sqrt{5}$ \\
\hline
$0$ & $1 - \sqrt{5}$ & $1 + \sqrt{5}$ \\
\hline
$0$ & $0$ & $1$ \\
\hline
$0$ & $-1 + \sqrt{5}$ & $1 + \sqrt{5}$ \\
\hline
$0$ & $1 + \sqrt{5}$ & $2$ \\
\hline
$0$ & $1 + \sqrt{5}$ & $-1 + \sqrt{5}$ \\
\hline
$0$ & $1$ & $0$ \\
\hline
$-1 + \sqrt{5}$ & $1 + \sqrt{5}$ & $0$ \\
\hline
$2$ & $1 + \sqrt{5}$ & $-1 + \sqrt{5}$ \\
\hline
$2$ & $-1 - \sqrt{5}$ & $-1 + \sqrt{5}$ \\
\hline
$2$ & $1 - \sqrt{5}$ & $1 + \sqrt{5}$ \\
\hline
$2$ & $0$ & $1 + \sqrt{5}$ \\
\hline
$1$ & $-1$ & $1$ \\
\hline
$1$ & $1$ & $1$ \\
\hline
$1 + \sqrt{5}$ & $2$ & $0$ \\
\hline
$3 + 3\sqrt{5}$ & $3 + \sqrt{5}$ & $2$ \\
\hline
$1$ & $0$ & $0$ \\
\hline

\end{tabular}
}
    \caption{The homogeneous coordinates of the 21-vertex graph $G_X$ that achieves the value $\Gamma(G_X)=32/107\approx0.2991$.}
    \label{tab:GFCN X21 coordinates}
\end{table}

\clearpage

\end{document}